\documentclass{aims}
\usepackage{amsmath}
  \usepackage{paralist}
  \usepackage{graphics} 
\usepackage{float}
\usepackage{amsxtra,latexsym,amsthm, amssymb, amscd}
  \usepackage{epsfig} 
\usepackage{graphicx}  \usepackage{epstopdf}
 \usepackage[colorlinks=true]{hyperref}
\hypersetup{urlcolor=blue, citecolor=red}

\newtheorem{theorem}{Theorem}[section]

\newtheorem{lemma}[theorem]{Lemma}

\theoremstyle{definition}
\newtheorem{definition}[theorem]{Definition}
\newtheorem{remark}{Remark}

\allowdisplaybreaks

\title[Stochastic forest model] 
      {A Sustainability Condition for Stochastic Forest Model}

\author[T\d{a}, Nguyen \&  Yagi]{}


 \email{tavietton[at]agr.kyushu-u.ac.jp}
 \email{nth.linh[at]ist.osaka-u.ac.jp}
 \email{atsushi-yagi[at]ist.osaka-u.ac.jp}

\thanks{$^\dag$ This work was supported by JSPS KAKENHI Grant Number 20140047.}
\thanks{$^\ddag$ This work was supported by JSPS Grant-in-Aid for Scientific Research (No. 26400166).}

\begin{document}
\maketitle

\centerline{\scshape T$\hat{\text o}$n Vi$\hat{\d{e}}$t  T\d{a}$^\dag$}
\medskip
{\footnotesize
 \centerline{Promotive Center for International Education and Research of Agriculture }
   \centerline{Faculty of Agriculture, Kyushu University}
   \centerline{6-10-1 Hakozaki, Nishi-ku, Fukuoka 812-8581, Japan}
} 

\medskip

\centerline{\scshape Linh Thi Hoai Nguyen}
\medskip
{\footnotesize
 \centerline{Department of Information and Physical Sciences}
   \centerline{Graduate School of Information Science and Technology, Osaka University}
   \centerline{1-5 Yamadaoka, Suita, Osaka 565-0871, Japan}
}

\medskip

\centerline{\scshape Atsushi Yagi$^\ddag$}
\medskip
{\footnotesize
 \centerline{Department of Applied Physics, Graduate School of Engineering, Osaka University}
   \centerline{1-5 Yamadaoka, Suita, Osaka 565-0871, Japan}
}

\bigskip


\begin{abstract}
A stochastic forest model of young and old age class trees is studied. First, we prove existence, uniqueness  and boundedness of global nonnegative solutions. Second,   we investigate asymptotic behavior of  solutions by giving a sufficient condition for sustainability of the forest. Under this condition, we  show existence  of  a Borel invariant measure. Third, we present several sufficient conditions for decline of the forest.  Finally, we give some numerical examples.
\end{abstract}

\section{Introduction} \label{Sec1}
\label{Intro}
In 1975, Antonovsky  \cite{MYA} introduced  a mono-species forest model  with two age classes of trees:
\begin{equation} \label{E0}
\begin{cases}
\frac{d u}{dt} =\rho v-\gamma(v) u-fu, \\
\frac{dv }{dt} =fu-hv.
\end{cases}
\end{equation}
Here, $u$ and $ v$ denote the tree densities of young and old age classes, respectively. The parameters $\rho, h$ and $f$ are  a reproduction rate, mortality of old  trees, and aging rate of young  trees, respectively; while $\gamma(v)=a(v-b)^2+c$ is a mortality of young  trees,  which is allowed to depend on the old-tree density. In addition, $a, b, c, \rho, f$ and $h$ are assumed to be  positive constants.

It is not difficult to see that for any pair $(u_0,v_0)$ of  nonnegative initial values $u_0$ and $ v_0$,  the  system \eqref{E0} possesses a nonnegative and global solution. Furthermore, \eqref{E0} possesses 
nonnegative stationary solutions  given by
\begin{enumerate}  
\item  $O =(0,0)$
\item  $ P^-=\Big(\frac{h}{f}(b-\sqrt{\frac{\rho f-h(c+f)}{ah}}), b-\sqrt{\frac{\rho f-h(c+f)}{ah}}\Big) $\hspace{2cm} (if $h\in [h_*, h^*])$
\item $ P^+=\Big(\frac{h}{f}(b+\sqrt{\frac{\rho f-h(c+f)}{ah}}), b+\sqrt{\frac{\rho f-h(c+f)}{ah}}\Big) $\hspace{2cm}  (if $h\in (0, h^*])$
\end{enumerate}
where $h_*=\frac{\rho f}{ab^2+c+f}, h^*=\frac{\rho f}{c+f}.$
The stability and instability of these solutions  depends strongly on the magnitude of the  mortality $h$  of old age class trees (see Table \ref{table1}).

 \begin{table}[ht]\label{table1}
{     \begin{tabular}{|r|c|c|r|}
     \hline
  $ h$  &$(0, h_*)$ & $(h_*,h^*)$&$(h^*, \infty)$\hspace*{1cm}{}\\ \hline
     $O$ &unstable & stable  & glob. asymp. stable \\ \hline
        $P^+$ &stable &   stable & $-$\hspace*{1.5cm}{}\\ \hline
        $P^-$ & $-$ &   unstable  &$-$\hspace*{1.5cm}{}\\ \hline
                    \end{tabular} }
\vspace{0.4cm} 
\caption{Stability and instability of stationary solutions of \eqref{E0}}
     \end{table}

On the basis of  \eqref{E0}, Kuznetsov et al. \cite{KABA}  introduced a mathematical model of mono-species forest with two age classes which takes into account the seed production and dispersion.  
The third author studied that model with his colleagues (see, e.g.,  \cite{Chuan1,Chuan2,Chuan3,Shirai} and \cite[Chapter 11]{AA}). It is  shown that  $h$  plays a crucial role in the asymptotic behavior of solutions.  

In the real world, the parameters in the model may be random variables   due to   unpredictability resulting from environmental, ecological and biological disturbances. In principle, the deterministic forest model can not handle  randomness.  
 Investigating the role of fluctuation of parameters  by using stochastic models should be an interesting problem in  environmental and ecological sciences.

As mentioned above, the asymptotic behavior of solutions to the deterministic forest model depends strongly on  the magnitude of $h$. Therefore,  in this paper we restrict ourselves   to consider  a stochastic forest model,  where $h$ is perturbed by  (Gaussian)  white noise. Since Gaussian white noises can be expressed as the generalized derivative of a Brownian motion, we  make a substitution: 
 $$h \leadsto h - \sigma dw_t$$
in \eqref{E0}, 
 where $\{w_t, t\geq 0\}$ is a one-dimensional Brownian motion  defined on a  filtered complete probability space $(\Omega, \mathcal F,\{\mathcal F_t\}_{t\geq 0}, \mathbb P)$,  and $\sigma>0$  is  the intensity of  the white noise. Our stochastic forest model is then 
 formulated by  It\^o stochastic differential equations in $\mathbb R^2$:
\begin{equation} \label{E1}
\begin{cases}
d u_t =\{\rho v_t-[a(v_t-b)^2+c+f] u_t\} dt, \\
dv_t  =(fu_t-hv_t)dt+\sigma v_t dw_t.
\end{cases}
\end{equation}

In this paper, we study the stochastic forest model \eqref{E1}. We prove existence of unique global solutions to  \eqref{E1} and then study their asymptotic behavior. On one hand, we  present   a sufficient condition for sustainability of  the forest. Under this condition, we  also prove existence of  a non-trivial Borel invariant measure. On the other hand, we  give several sufficient conditions for decline of the forest. The results are illustrated by a few numerical examples.

To prove existence of non-trivial invariant measures to \eqref{E1}, a common method is to find four one-dimensional processes, namely $u^1, u^2, v^1$ and $v^2$, which satisfy two conditions:
\begin{itemize}
\item [{\rm (i)}] $u_t$ and $v_t$ are bounded  by these processes, i.e. $u^1(t)<u_t<u^2(t)$ and $v^1(t)<v_t<v^2(t)$ for $ 0<t<\infty$ 
\item [{\rm (ii)}] These four processes do not hit the boundaries in the sense that there exist $\epsilon>0$ and $M>0$ such that
\begin{align*}
\begin{cases}
\epsilon<u^1(t)<u^2(t)<M &\hspace{1cm} \text{a.s.}, 0<t<\infty,\\
\epsilon<v^1(t)<v^2(t)<M & \hspace{1cm} \text{a.s.}, 0<t<\infty\\
\end{cases}
\end{align*}
\end{itemize}
However, this can not be done because 
  $$\liminf_{t\to \infty} v_t=0 \hspace{1cm} \text{ a.s. (see Theorem \ref{T2})}. $$

 To overcome this difficulty, we use the semigroup method presented in \cite{Foguel,Mic}. First, we establish some estimates for the average of integrals of solutions (see \eqref{estimate1} and \eqref{estimate2}). Then, we construct a strongly continuous semigroup generated by  solutions of  \eqref{E1}. Using these estimates and a theorem in  \cite{Mic}, we show that the semigroup enjoys an invariant measure.

The organization of the present paper is as follows. Section \ref{Sec2} proves existence and  boundedness of unique global nonnegative  solutions to \eqref{E1}. Section \ref{Sec3}  investigates sustainability of  the forest and existence of a Borel invariant measure. To the contrary, Section \ref{Sec4}  presents some sufficient conditions  for decline of the forest. Finally,  Section \ref{Sec5} gives some numerical examples.

\section{Global   solutions}\label{Sec2}
In this section, we prove existence of unique global nonnegative  solutions to \eqref{E1} and show boundedness of  solutions.

Put $\mathbb R^2_+=\{(u,v); u> 0, v> 0\}.$ Denote by $\overline{\mathbb R^2_{+}}$ the  closure of $\mathbb R^2_+$ in $\mathbb R^2$. Put 
$$M_0=\inf \{u; \rho v- [\gamma(v)+f]u < 0 \quad \text { for every } v>0\}.$$
Then,
\begin{equation}\label{E4}
\begin{aligned}
M_0&=\inf \Big\{u;  \frac{\rho v}{\gamma(v)+f}-u < 0 \quad \text { for every } v>0\Big\}\\
&=\sup_{v>0} \frac{\rho v}{a(v-b)^2+c+f}\\
&=\frac{\rho\sqrt{ab^2+c+f}}{\sqrt{a}\left[(\sqrt{ab^2+c+f}-\sqrt a b)^2+c+f\right]}.
\end{aligned}
\end{equation}
For biological reasons, throughout this paper, initial values  for  \eqref{E1} are  taken from $\overline{\mathbb R^2_{+}}$.

Let us first prove existence of unique global nonnegative  solutions to \eqref{E1}. 
We use the following lemma.

\begin{lemma}\label{lem6}
Consider the one-dimensional stochastic differential equation:
\begin{equation} \label{E15}
\begin{cases}
dx_t=(a_2-a_1x_t)dt+\alpha x_tdw_t,\\
x_t|_{t=0}=x_0>0,
\end{cases}
\end{equation}
where $a_1, a_2$ and $\alpha $ are positive constants.
Then,  there exists a unique global solution of \eqref{E15} such that 
\begin{itemize}
  \item  [{\rm (i)}] For  $0\leq a_3<\infty,$
$$\lim_{t\to\infty}\frac{\log (a_3+x_t)}{t}=0 \hspace{1cm} \text{  a.s.}$$
  \item   [{\rm (ii)}] For  $1\leq \theta<\infty$ and $0<T<\infty,$ there exists $\alpha_{\theta,T}$ depending on $\theta $ and $T$ such that 
       $$ \mathbb E x^\theta_t\leq \alpha_{\theta,T}, \hspace{1cm} 0\leq t \leq T.$$    
  In addition, if $1\leq \theta< 1+\frac{2a_1}{\alpha^2}$, then $\alpha_{\theta,T}$  is independent of $T$
  \item  [{\rm (iii)}] $\limsup_{t\to \infty} x_t=\infty$ and $  \liminf_{t\to \infty} x_t=0$  \hspace{0.3cm}\text{  a.s.}
 \end{itemize}
\end{lemma}
Since the proof of the lemma  is quite easy, we omit it.

\begin{theorem} \label{T1}  
Let $(u_0,v_0) \in \overline{\mathbb R^2_+}.$ Then, there exists a unique global solution $(u_t, v_t)$ of \eqref{E1}   such that  $(u_t, v_t)|_{t=0}=(u_0,v_0)$ and 
$$ (u_t,v_t)\in \overline{\mathbb R^2_+} \hspace{1cm} \text{  a.s.,   } 0<t<\infty.$$
In addition, if $u_0+v_0>0,$ then 
$$ (u_t,v_t)\in \mathbb R^2_+ \hspace{1cm} \text{  a.s.,  } 0<t<\infty.$$
\end{theorem}

\begin{proof}
Since all the functions on the right-hand side of \eqref{E1} are locally Lipschitz continuous,  there is a unique local solution $(u_t,v_t)$ defined on an interval $[0,\tau),$ where $\tau$ is a stopping time having the following property (see, e.g.,  \cite{Arnold,Friedman}).  If $\mathbb P\{\tau<\infty\}>0,$ then $\tau$  is an explosion time on $\{\tau<\infty\}$, i.e.   
$$\lim_{t\to\tau}(|u_t|+|v_t|)=\infty \hspace{1cm} \text{ a.s.  on } \{\tau<\infty\}.$$
Therefore, it suffices to show that $\tau=\infty$ a.s. and that  $u_t\geq 0, v_t\geq 0$ a.s. for  $0<t<\infty$. 
To prove this,  we use the method   in \cite{LinhTon,TonLinhYagi}.

Consider the four cases of initial values.

{\bf Case 1}: $u_0=v_0=0.$ This is a trivial case, since  $u_t=v_t=0$ a.s. for  $0\leq t<\infty$.

{\bf Case 2}:  $(u_0,v_0)\in \mathbb R^2_+$. \\
 Let $k_0>0$ be  a positive integer such that  $u_0$ and $v_0$  lie in the  interval $[\frac{1}{k_0},k_0]$. 
Denote 
\begin{align*}  
H_k=[\frac 1k, k]\times [\frac 1k, k], \hspace{1cm} k=1,2,\dots,
\end{align*}
 then $\cup_{k=k_0}^{\infty} H_k=\mathbb R^2_+$. 
Let us define a sequence $\{\tau_k\}_{k=k_0}^\infty$ of stopping times  by 
\begin{equation}  \label{DefineTauk}
\tau_k=\inf \left\{0<t<\tau;  (u_t,v_t)\notin H_k\right\}
\end{equation}
with the convention  $\inf\emptyset=\infty$. 
It is obvious that  $\{\tau_k\}_{k=k_0}^\infty$  is nondecreasing.  Hence, there exists a limit $\tau_\infty$ of this sequence as $k \rightarrow  \infty$:
\begin{equation*}   
\tau_{\infty}=\lim_{k \rightarrow \infty} \tau_k\leq \tau \hspace{1cm} \text{a.s.}
\end{equation*}

Let us prove that $\tau_\infty=\infty$ a.s. Indeed, suppose the contrary, then  there would exist  $T>0$ and $0<\epsilon<1$ such that  
\begin{equation}   \label{E2.5}
\mathbb P\{\tau_\infty<T\}>\epsilon.
\end{equation}
Consider a positive function $V$ on $\mathbb R^2_+$, which is defined by
$$V(u, v)=u^2+v^2-\log u-\log v, \hspace{1cm} u>0, v>0.$$ 
The It\^o  formula gives 
\begin{align*}
dV(u_t, v_t)=[LV](u_t,v_t) dt + (2\sigma v^2-\sigma ) dw_t,
\end{align*}
where the infinitesimal operator $L$  is given by
\begin{equation} \label{E3}
[L\cdot](u, v)=\frac{1}{2}\sigma^2 v^2 \frac{\partial^2 \cdot}{\partial v^2} + [\rho v-\{\gamma(v)+f\} u]\frac{\partial \cdot}{\partial u} +(fu-hv)\frac{\partial \cdot}{\partial v}.
\end{equation}

It is possibly seen that
\begin{align*}
\begin{aligned}
\,[LV](u, v)=&2(\rho+f)uv-2[\gamma(v)+f]u^2+(\sigma^2-2h)v^2+\gamma(v)\\
&-\frac{\rho v}{u}-\frac{fu}{v}+f+h+\frac{\sigma^2}{2}.
\end{aligned}
\end{align*}
In addition,  there exist $M_i>0 \,(i=1,2)$ such that 
$$[LV](u,v)<M_1 V(u,v)+ M_2, \hspace{1cm} (u,v)\in \mathbb R^2_+.$$
 Therefore, 
$$\int_0^{t\wedge \tau_k}d V(u_s, v_s)\leq \int_0^{t\wedge \tau_k} [M_1 V(u_s,v_s)+M_2]ds+\int_0^{t\wedge \tau_k}(2\sigma v_s^2-\sigma ) dw_s. $$

Taking  the expectations of the two sides of the latter inequality, we obtain that 
\begin{align*}
&\mathbb E V(u_{t\wedge \tau_k},v_{t\wedge \tau_k}) \\
&\leq V(u_0,v_0)+M_2(t\wedge \tau_k)+M_1\mathbb E\int_0^{t\wedge \tau_k} V(u_s,v_s)ds\\
&\leq [V(u_0,v_0)+M_2T] +M_1 \int_0^t \mathbb EV(u_{s\wedge \tau_k},v_{s\wedge \tau_k})ds, \hspace{1cm}0\leq t\leq T.
\end{align*}
  The Gronwall inequality then provides that
\begin{align} 
\mathbb E V(u_{t\wedge \tau_k},v_{t\wedge \tau_k}) &\leq [V(u_0,v_0)+M_2T]  e^{M_1 t}\notag\\
&\leq [V(u_0,v_0)+M_2T]  e^{M_1 T}, \hspace{1cm} 0\leq t\leq T. \label{Eq6}
\end{align}
Hence, 
\begin{align}
 [V(u_0,v_0)+M_2T]  e^{M_1T}&\geq \mathbb E V(u_{T\wedge \tau_k},v_{T\wedge \tau_k})     \notag\\
 &\geq \mathbb E[{\bf 1}_{\{\tau_\infty<T\}} V(u_{T\wedge \tau_k},v_{T\wedge \tau_k})]   \notag\\
&=\mathbb E[{\bf 1}_{\{\tau_\infty<T\}} V(u_{\tau_k},v_{\tau_k})].  \label{Eq6.1}
\end{align}

On the other hand,  \eqref{DefineTauk} gives 
\begin{align}
V(u_{\tau_k},v_{\tau_k})&\geq \min\Big\{k^2-\log k,  \Big(\frac1k\Big)^2-\log \Big(\frac1k\Big)\Big\}    \notag\\
&= \min\Big\{k^2-\log k,  \log k+\frac{1}{k^2}\Big\} \hspace{1cm}\text{ on } \{\tau_k<\infty\}.   \label{Eq6.2}
\end{align}

Thanks to \eqref{E2.5}, \eqref{Eq6.1} and \eqref{Eq6.2}, we observe that 
$$[V(u_0,v_0)+M_2T]  e^{M_1 T} \geq \epsilon \min\Big\{k^2-\log k,  \log k+\frac{1}{k^2}\Big\}.
$$
Letting $k\to\infty$, we arrive at a contradiction:
 $$\infty>[V(u_0,v_0)+M_2T]  e^{M_1 T}=\infty.$$
 Thus,  $\tau_\infty=\tau=\infty$ a.s. Furthermore,   $(u_t,v_t)\in \cup_{k=k_0}^{\infty} H_k=\mathbb R^2_+$  a.s. for  $0\leq t<\infty$.

{\bf Case 3}:  $u_0>0, v_0=0.$ \\
Let $k_0>0$ be a positive integer  such that  $u_0$   lies in $[\frac{1}{k_0},k_0]$. Denote 
$$H_k^1=[\frac 1k, k] $$
and 
$$\tau_k=\inf \left\{0\leq t<\tau;  u_t\notin H_k^1\right\}, \hspace{1cm}  k=k_0, k_0+1,\dots$$
Clearly, the sequence $\{\tau_k\}_{k=k_0}^\infty$ has a limit $\tau_{\infty}$ as $k\to \infty$:
 $$\tau_{\infty}=\lim_{k \rightarrow \infty} \tau_k\leq \tau \hspace{1cm} \text{ a.s.}$$

Let us first show that 
\begin{equation} \label{Eq4}
 v_t>0 \hspace{1cm} \text{ a.s.,  } 0<t\leq \tau_k.
\end{equation}
Indeed, 
 due to the comparison theorem (see  \cite{WI}),  $v_t>{\underline v}_t$ a.s. for  $0<t\leq \tau_k$, where ${\underline v}_t$ is the solution of this equation:
\begin{align*}
\begin{cases}
d{\underline v}_t=-h {\underline v}_t dt+\sigma {\underline v}_t dw_t,\\
{\underline v}|_{t=0}=v_0=0.
\end{cases}
\end{align*}
Obviously, ${\underline v}_t=0.$ Thus, \eqref{Eq4} follows.

Let us now verify that  
\begin{equation} \label{Eq3}
\begin{aligned}
&\,\text{ there exists } \alpha>0 \text{ such that }
 \mathbb Ev_{t\wedge \tau_k}\leq \alpha \text{ for }  0\leq t<\infty. \text{ In addition, for} \\
 &\hspace{0.2cm} \text{any }  T>0, \text{ there exists } \beta>0 \text{ such that } \mathbb E\gamma (v_{t\wedge \tau_k}) \leq \beta \text {  for } 0\leq t \leq T. 
 \end{aligned}
\end{equation}
Indeed, by virtue of  the first equation of  \eqref{E1},  
\begin{align*} 
\frac{du}{dt}&=[\gamma (v) +f]\Big[\frac{\rho v}{\gamma(v)+f} -u\Big]  \\
&\leq[\gamma (v) +f](M_0 -u), \hspace{1cm} 0\leq t<\tau,
\end{align*}
where $M_0$ is defined in \eqref{E4}.
Solving this differential inequality, we obtain that   
\begin{equation}\label{Eq2}
u_t\leq M_0-(M_0 -u_0)e^{-\int_0^t [\gamma(v)+f]ds},  \hspace{1cm} 0\leq t<\tau.
\end{equation}
Hence, 
\begin{equation} \label{Eq7}
u_t \leq \max\{u_0, M_0\} \hspace{1cm} \text{   a.s.,   } 0\leq t<\tau.
\end{equation}

Using \eqref{Eq7} and applying the comparison theorem  for the second equation of \eqref{E1}, we   observe that  
\begin{equation} \label{Eq7.1}
v_t\leq \bar v_t \hspace{1cm} \text{a.s.,  }  0\leq t<\tau,
\end{equation}
 where $\bar v_t$ is the solution of the one-dimensional stochastic differential equation: 
\begin{equation}\label{E4.1}
\begin{cases}
d\bar v_t=[\max\{u_0, M_0\}-h\bar v_t] dt+\sigma \bar v_tdw_t,\\
\bar v_t|_{t=0}=v_0.
\end{cases}
\end{equation}
Thanks to Lemma \ref{lem6}--(ii) and   \eqref{Eq7.1}, it is easily seen that \eqref{Eq3}   holds true.

Let us finally observe that  
$$\tau=\infty  \quad \text{  and } \quad (u_t,v_t)\in \mathbb R^2_+  \hspace{1cm} \text{ a.s., } 0<t<\infty.$$
  In view of  \eqref{Eq4}, it suffices to show that $\tau_\infty=\infty$. Indeed, suppose the contrary, then  there would exist  $T>0$ and $0<\epsilon<1$ such that  
$$\mathbb P\{\tau_\infty<T\}>\epsilon.$$
Consider a positive function $H$ on $(0,\infty)$ defined by
$$H(u)=u^2 -\log u, \hspace{1cm} u>0.$$
The It\^o  formula then gives 
\begin{align*}
\int_0^{t\wedge \tau_k}d H(u_s)&= \int_0^{t\wedge \tau_k} [2\rho u_sv_s -2u_s^2\gamma(v_s)-\frac{\rho v_s}{u_s}+ \gamma(v_s)]ds \\
&\leq \int_0^t [2\rho u_{s\wedge \tau_k}v_{s\wedge \tau_k} + \gamma(v_{s\wedge \tau_k})]ds.
\end{align*}
Taking the expectation of the two sides of this inequality and using \eqref{Eq3} and \eqref{Eq7}, we observe that 
\begin{align}
\mathbb E H(u_{s\wedge \tau_k})&\leq \int_0^t (2\rho \max\{u_0, M_0\} \alpha + \beta) ds\notag \\
&\leq  (2\rho \max\{u_0, M_0\} \alpha + \beta)  T, \hspace{1cm} 0\leq t \leq T. \label{Eq8}
\end{align}

By using \eqref{Eq8}  instead of  \eqref{Eq6},  we repeat the same argument as in  Case 2 to conclude that $\tau_\infty=\infty$ a.s. 

{\bf Case 4}:  $u_0=0, v_0>0.$ The proof for this case is similar to one for Case 3. 

By the above arguments, the proof of the theorem is  complete.
\end{proof}

Let us now show  boundedness  for the density $u$ of young age class trees  and for  moments of the density  $v$ of old age class trees. 
\begin{theorem} \label{T2}
Let $(u_t,v_t)$ be a solution of \eqref{E1} such that  $(u_t,v_t)|_{t=0}=(u_0,v_0) \in \overline{\mathbb R^2_+}.$ Then,
\begin{itemize}
  \item [ \rm {(i)}] $\sup_{0\leq t<\infty}u_t \leq M^*$ \,   a.s., where $M^*= \max\{u_0, M_0\} $  and $M_0$ is defined in \eqref{E4}
\item [ \rm {(ii)}]
$\limsup_{t\to \infty} u_t\leq M_0$ \hspace{1cm} a.s.
\item [ \rm {(iii)}] For any  $1\leq \theta < 1+\frac{2h}{\sigma^2}, $  there exists  $\alpha_\theta>0$ such that 
$$\limsup_{t\to \infty} \mathbb E v^\theta_t\leq \alpha_\theta$$
\item[ \rm {(iv)}] $\liminf_{t\to \infty} v_t=0$  \,  a.s.
\end{itemize}
\end{theorem}
\begin{proof}
Clearly,   (i) and (ii)  follow from \eqref{Eq2} and  \eqref{Eq7}.
 Meanwhile, applying  Lemma \ref{lem6} to the equation \eqref{E4.1} and using the fact that $v_t\leq \bar v_t,$   (iii) and (iv) follow.
\end{proof}

\section{Sustainability of forest} \label{Sec3}
In this section, we  present a sufficient condition for  sustainability of the forest. Under this condition,  we also show existence of a Borel invariant measure on $\mathbb R^2_+$ for the system \eqref{E1}.

\subsection{Sustainability condition}
Let us show that if the intensity of noise and  mortality of old age class trees are small enough, then the forest is sustainable.
\begin{definition}
The system \eqref{E1}  is said to be sustainable if for every initial value $(u_0,v_0)\in \overline{\mathbb R^2_+}\backslash\{(0,0)\}$, the solution $(u_t,v_t)$  satisfies 
$$\limsup_{t\to\infty} \mathbb E u_t>0 \quad \text{ and } \quad  \limsup_{t\to\infty} \mathbb E v_t>0.$$
\end{definition}

\begin{theorem}\label{T8}
Assume that
\begin{equation}  \label{sustainablecondition}
h<\frac{\rho f}{ab^2+c+f} \quad   \text{ and } \quad \sigma^2<2(\frac{\rho f}{ab^2+c+f}-h)
\end{equation}
and $(u_0,v_0)\in \overline{\mathbb R^2_{+}}\backslash\{(0,0)\}$. Let $(u_t,v_t)$ be the solution of  \eqref{E1} with $(u_t,v_t)|_{t=0}=(u_0,v_0)$.    Then,  there exists $\epsilon>0$ which is independent of $(u_0,v_0)$ such that 
\begin{equation} \label{estimate1}
 \liminf_{t\to\infty}\frac{1}{t}\int_0^tv_sds>\epsilon,  \quad \quad \liminf_{t\to\infty}\frac{1}{t}\int_0^tv^2_sds>\epsilon \hspace*{1cm} \text{ a.s.}
\end{equation} 
and 
\begin{equation}\label{estimate2}
 \liminf_{t\to\infty}\frac{1}{t}\int_0^t\mathbb E u_sds>\epsilon,  \quad\quad\liminf_{t\to\infty}\frac{1}{t}\int_0^t\mathbb E v_sds>\epsilon.
\end{equation} 
As a consequence,  \eqref{E1} is sustainable.
\end{theorem}

\begin{proof}
By virtue of  \eqref{sustainablecondition}, there exists a constant $\kappa>0$ such that 
\begin{align*}
\begin{cases}
\frac{ab^2+c+f}{f}<\kappa<\frac{\rho}{h},\\
\sigma^2< \frac{2(\rho-\kappa h)}{\kappa}.
\end{cases}
\end{align*}
Consider  a function $Q$ on $ \overline{\mathbb R^2_+}\backslash\{(0,0)\} $ defined by 
$$Q(u,v)=\log (u+\kappa v).$$
Theorem \ref{T1} and the It\^o formula then provide that 
\begin{equation}\label{E22}
dQ(u_t, v_t)=[LQ](u_t,v_t)dt +\frac{\sigma \kappa v_t}{u_t+\kappa v_t}dw_t,
\end{equation}
where the operator $L$ is defined in  \eqref{E3}. After some simple calculations, we obtain that 
\begin{equation} \label{E22.1}
[LQ](u,v)=\frac{(\kappa f-c-f)u+(\rho-\kappa h)v}{u+\kappa v}-\frac{\sigma^2\kappa^2v^2}{2(u+\kappa v)^2}-\frac{au(v-b)^2}{u+\kappa v}.
\end{equation}
Thereby, by using the estimate (i) of Theorem \ref{T2}, we  observe that
\begin{align}
\log&(u_t+\kappa v_t)\notag\\
=&\log(u_0+\kappa v_0)-\int_0^t\frac{av_s^2u_s-2ab u_sv_s}{u_s+\kappa v_s}ds+\int_0^t\frac{\sigma \kappa v_s}{u_s+\kappa v_s}dw_s\notag\\
&+\int_0^t \Big[\frac{(\kappa f-c-f-ab^2)u_s+(\rho-\kappa h)v_s}{u_s+\kappa v_s}-\frac{\sigma^2\kappa^2v_s^2}{2(u_s+\kappa v_s)^2}\Big]ds\notag\\
\geq&\log(u_0+\kappa v_0)-aM^*\int_0^t\frac{v_s^2}{M^*+\kappa v_s}ds+\int_0^t\frac{\sigma \kappa v_s}{u_s+\kappa v_s}dw_s\notag\\
&+\int_0^t \Big[\frac{(\kappa f-c-f-ab^2)u_s+(\rho-\kappa h)v_s}{u_s+\kappa v_s}-\frac{\sigma^2\kappa^2v_s^2}{2(u_s+\kappa v_s)^2}\Big]ds.
\label{E23}
\end{align}

Let us show that there exists $\varepsilon_1>0$ such that  for all $(u,v)\in {\mathbb R^2_+},$
\begin{equation} \label{E24}
\frac{(\kappa f-c-f-ab^2)u+(\rho-\kappa h)v}{u+\kappa v}-\frac{\sigma^2\kappa^2v^2}{2(u+\kappa v)^2}\geq \frac{\varepsilon_1}{2}.
\end{equation} 
Indeed, \eqref{E24} is equivalent to that $F(u, v) \geq 0$ for all $(u,v)\in {\mathbb R^2_+}$, where
\begin{align*}
\begin{aligned}
F(u,v)=&[2(\kappa f-c-f-ab^2)-\varepsilon_1]u^2+2[\kappa (\kappa f-c-f-ab^2)\\
&+(\rho-\kappa h)-\kappa  \varepsilon_1]v u+[2\kappa (\rho-\kappa h)-\kappa^2\sigma^2-\kappa^2 \varepsilon_1]v^2.
\end{aligned}
\end{align*}
Since $\sigma^2< \frac{2(\rho-\kappa h)}{\kappa }$, it is easily seen that there exists a small $\varepsilon_1>0$ such that  the quadratic equation $F(u,v)=0$ in the variable $u$ has two non-positive solutions for every $v\geq 0$. Thus,  $F(u, v) \geq 0$  for all $(u,v)\in {\mathbb R^2_+}$.

{\it Proof for \eqref{estimate1}}.  Due to  \eqref{E23}, \eqref{E24} and the fact that $v_t\leq \bar v_t,$ where $\bar v_t$ is the solution of \eqref{E4.1}, we have
\begin{align}
\frac{\log(M^*+\kappa\bar v_t)}{t}&+\frac{aM^*}{t}\int_0^t\frac{v^2}{M^*+\kappa v}ds \notag\\
&\geq \frac{\log(u_0+\kappa v_0)}{t}+\frac{\varepsilon_1}{2}+\frac{1}{t}\int_0^t\frac{\sigma \kappa v}{u+\kappa v}dw_s.\label{E25}
\end{align}

Put 
\begin{equation} \label{E25.1}
N_t=\int_0^t\frac{ \sigma \kappa v_s}{u_s+\kappa v_s}dw_s.
\end{equation}
Then, $\{N_t\}_{0\leq t<\infty}$ is a real-valued continuous martingale vanishing at $t=0$. Furthermore, $\{N_t\}_{0\leq t<\infty}$  has  a quadratic form given by 
$$\langle N \rangle_t=\int_0^t \frac{ \sigma^2 \kappa^2v^2_s}{(u_s+\kappa v_s)^2}ds\leq \sigma^2 t.$$
The strong law of large numbers for martingale (see, e.g., \cite{IS,Mao}) then gives  
\begin{equation} \label{E25.2}
\lim_{t\to\infty} \frac{N_t}{t}=0 \hspace{1cm} \text{ a.s.}
\end{equation}

In the meantime, applying Lemma \ref{lem6}--(i) to the equation \eqref{E4.1} and using Theorem \ref{T2}--(ii), we observe that 
\begin{align*}
\begin{cases}
\lim_{t\to\infty} \frac{\log(M^*+\kappa \bar v_t)}{t}=0,\\
u_t\leq M^* \hspace{1cm} \text{ a.s., }  0\leq t<\infty.
\end{cases}
\end{align*}
Taking the limit as $t\to \infty$ of the two sides of  \eqref{E25}, we  hence obtain that
$$\liminf_{t\to\infty} \frac{1}{t}\int_0^t \frac{aM^*v_s^2}{M^*+\kappa v_s}ds\geq \frac{\varepsilon_1}{2} \hspace{1cm} \text{a.s.}$$
Since $\frac{v_s^2}{M^*+\kappa v_s}<\frac{1}{\kappa } v_s$ and $\frac{v_s^2}{M^*+\kappa v_s}<\frac{1}{M^*} v_s^2,$ we conclude that 
\begin{align*} 
\begin{aligned}
\begin{cases}
 \liminf_{t\to\infty}\frac{1}{t}\int_0^tv_sds>\frac{\kappa \varepsilon_1}{2aM^*} &\hspace*{1cm} \text{ a.s.},\\
   \liminf_{t\to\infty}\frac{1}{t}\int_0^tv_s^2ds>\frac{\varepsilon_1}{2a} &\hspace*{1cm} \text{ a.s.}
 \end{cases}
 \end{aligned}
\end{align*}
from which it follows  \eqref{estimate1}.

{\it Proof for  \eqref{estimate2}}. Taking the expectation of the two sides of  \eqref{E25}, we have
\begin{align*}
\frac{\varepsilon_1}{2}\leq &\liminf_{t\to \infty} \left[\frac{\log(M^*+\kappa \bar v_t)}{t}+\frac{aM^*}{t} \mathbb E \int_0^t \frac{v^2}{M^*+\kappa v}ds\right]\\
\leq &\liminf_{t\to \infty} \left[\frac{M^*+\kappa \mathbb E \bar v_t}{t}+\frac{aM^*}{\kappa } \frac{1}{t} \int_0^t \mathbb Evds\right],
\end{align*}
here we used the estimate 
$$\log (M^*+x)<M^*+x, \hspace{1cm} 0<x<\infty.$$
On account of Lemma \ref{lem6}, the solution $\bar v$ of \eqref{E4.1} satisfies the estimate 
$$\mathbb E\bar v_t\leq \alpha_1, \hspace{1cm}  0<t<\infty,$$
where $\alpha_1$ is some positive constant. We thus have shown that
$$\liminf_{t\to \infty} \frac{1}{t} \int_0^t \mathbb Ev_sds\geq \frac{\varepsilon_1 \kappa }{2aM^*}.$$

Meanwhile, taking the expectation of the two sides of  the second equation of \eqref{E1}, it follows that 
\begin{align*}
\liminf_{t\to\infty}\frac{f}{t}\int_0^t \mathbb Eu_s ds&=\liminf_{t\to\infty}\Big[\frac{\mathbb Ev_t}{t}+\frac{h}{t}\int_0^t \mathbb E v_sds\Big] \notag\\
&=\liminf_{t\to\infty}\frac{h}{t}\int_0^t \mathbb E v_sds\geq \frac{h\kappa \varepsilon_1}{2aM^*}. 
\label{sustainableproof2}
\end{align*}
Therefore, \eqref{estimate2} has been verified. 
 As a consequence,   
$$\limsup_{t\to \infty} \mathbb E u_t>0$$
 and 
$$\limsup_{t\to \infty} \mathbb E v_t>0.$$ 
This means that  \eqref{E1} is sustainable. We complete the proof.
\end{proof}

\subsection{Existence of Borel invariant measure}
Let us   show existence of a Borel invariant measure of the It\^o  process $(u_t,v_t)$,  which concentrates on some domain of $\mathbb R^2_+$   under the assumptions in Theorem \ref{T8}. 

Let $P(\cdot,\cdot,\cdot,\cdot)$ be the transition probability of  $(u_t,v_t)$: 
$$P(t, x,y, K)=\mathbb P\{ (u_t,v_t)\in K; (u_t,v_t)|_{t=0}=(x,y)\},   $$ 
for $0\leq t<\infty, (x,y)\in \overline{\mathbb R^2_+},$ and $  K\in \mathcal B(\overline{\mathbb R^2_+}).$
 It is well known that  (see, e.g., \cite{Foguel,Mic})
\begin{itemize}
  \item [(i)] $P(t, x,y, \cdot)$  induces a strongly continuous semigroup $\{P_t\}_{0\leq t<\infty}$ of operators    on the space $C_B(\overline{\mathbb R^2_+})$ of  bounded continuous functions:
$$P_tf(x,y)=\int_{\mathbb R^2_+} f(\xi,\eta)P(t, x,y, d\xi d\eta), \hspace{1cm} f\in C_B(\overline{\mathbb R^2_+})$$
\item [(ii)] $P(t, x,y, \cdot)$  induces a positive contraction $[\cdot P_t]$ on the space $M(\overline{\mathbb R^2_+},  \mathcal B(\overline{\mathbb R^2_+}))$ of finite signed measures:
$$[\mu P_t](K)=\int_{\mathbb R^2} P(t, x,y, K) \mu(dxdy),   \hspace{1cm} \mu \in M(\overline{\mathbb R^2_+}, \mathcal B(\overline{\mathbb R^2_+})), K\in \mathcal B(\overline{\mathbb R^2_+})$$
\end{itemize}

\begin{definition}
A Borel  measure $\nu$ on $\overline{\mathbb R^2_+}$ (i.e. a positive measure which is finite on any compact set of $\overline{\mathbb R^2_+}$) is said to be  invariant  with respect to $\{P_t\}_{0\leq t<\infty}$  if for  $0<t<\infty$ and $K\in \mathcal B(\overline{\mathbb R^2_+}),$
$$[\nu P_t](K) =\nu (K).   $$
\end{definition}
The following result is well known.

\begin{theorem}[Michael  {\cite[Theorem 5.7]{Mic}}] \label{foguel}
Let $X$ be a locally compact perfectly normal topological space. Let $\{Q_t\}_{0\leq t<\infty}$ be a strongly continuous semigroup on $C_B(X)$ generated by a transition probability on $(X,\mathcal B(X)).$ If there exists a nonnegative function $ g$ in the space $ C_0(X)$ of continuous functions with compact support  such that  
$$\int_0^\infty Q_t g(x) dt=\infty,����   \hspace{1cm}  x\in X,$$ 
then there exists a Borel invariant measure for $\{Q_t\}_{0\leq t<\infty}$.
\end{theorem}

We are now ready to state our theorem.
\begin{theorem}
Let \eqref{sustainablecondition} be satisfied.  Then, $\{P_t\}_{0\leq t<\infty}$ has a Borel invariant measure  which concentrates on some domain of $\mathbb R^2_+\cap \{(u,v); u\leq M_0\}$.
\end{theorem}
\begin{proof}
To prove this theorem, we  construct a function $g\in  C_0(\overline{\mathbb R^2_+})$ which satisfies the assumption in Theorem \ref{foguel}. 

On  account of  Theorem \ref{T8}, we have 
$$ \liminf_{t\to\infty} \frac{1}{t}\int_0^t {\bf 1}_{\{v\geq \frac{ \epsilon}{2}\} } vds\geq \frac{ \epsilon}{2}>0 \hspace{1cm} \text{ a.s.}$$
Using Theorem \ref{T2}--(iii)  and the H\"{o}lder inequality,  for any  $0\leq \theta<\frac{2h}{\sigma^2},$ there exists $n_\theta>0$ such that
\begin{align*}
\frac{ \epsilon}{2}\leq&\liminf_{t\to\infty} \frac{1}{t}\int_0^t \mathbb E[{\bf 1}_{\{v_s\geq \frac{ \epsilon}{2}\} } v_s]ds\\
\leq&\liminf_{t\to\infty} \left\{ \left[\frac{1}{t}\int_0^t \mathbb E{\bf 1}_{\{v_s\geq \frac{ \epsilon}{2}\}}ds\right]^{\frac{\theta}{\theta+1}} \left[\frac{1}{t}\int_0^t \mathbb Ev_s^{1+\theta}ds\right]^{\frac{1}{\theta+1}}\right\}\\
\leq&n_\theta\liminf_{t\to\infty}  \left(\frac{1}{t}\int_0^t \mathbb P\{v_s\geq\frac{ \epsilon}{2}\}ds\right)^{\frac{\theta}{\theta+1}}.
 \end{align*}
Thereby, there exists $\epsilon_0>0$ such that 
 \begin{equation} \label{E26}
 \liminf_{t\to\infty}  \frac{1}{t}\int_0^t \mathbb P\{v_s\geq\frac{ \epsilon}{2}\}ds > \epsilon_0.
 \end{equation}

 On the other hand, by Theorem \ref{T2}--(iii), there exists $\alpha>0$ such that $ \mathbb E v_t\leq \alpha$ for  $0\leq t<\infty$. The Markov inequality then provides that
 $$ \mathbb P\left \{v_t\geq  \frac{2\alpha}{\epsilon_0}\right\}\leq  \frac{\epsilon_0}{2\alpha} \mathbb E v_t\leq \frac{\epsilon_0}{2}.$$
Hence, 
 \begin{equation}\label{E27}
 \inf_{0\leq t<\infty} \mathbb P\left \{v_t<  \frac{2\alpha}{\epsilon_0}\right\}\geq 1-\frac{\epsilon_0}{2}.
 \end{equation}

 Put 
$$K=\{(u,v);  0\leq u\leq M^*, \frac{\epsilon_0}{2}\leq v\leq \frac{2\alpha}{\epsilon_0}\}.$$
 Let us show that 
 \begin{equation} \label{E28}
 \liminf_{t\to\infty}\frac{1}{t}\int_0^t \mathbb P\{(u_s,v_s)\in K\}ds\geq \frac{\epsilon_0}{2}>0.
 \end{equation}
  Indeed, suppose the contrary, then there would exist an increasing sequence $\{t_n\}$ such that $t_n\to \infty$ as $n\to \infty$ and 
$$\frac{1}{t_n}\int_0^{t_n} \mathbb P \{(u_s,v_s)\in K\}ds< \frac{\epsilon_0}{2}, \hspace{1cm}   n=1,2,3,\dots$$
 Theorem \ref{T2}--(i) and  \eqref{E27} then give
\begin{align*}
&\frac{1}{t_n}\int_0^{t_n} \mathbb P\left\{v_s<  \frac{\epsilon}{2}\right\}ds\\
&=\frac{1}{t_n}\int_0^{t_n} \mathbb P\left\{v_s<  \frac{2\alpha}{\epsilon_0}\right\}ds
-\frac{1}{t_n}\int_0^{t_n} \mathbb P\left\{v_s\in \Big[\frac{\epsilon}{2},\frac{2\alpha}{\epsilon_0}\Big] \right\}ds\\
&= \frac{1}{t_n}\int_0^{t_n} \mathbb P\left\{v_s<  \frac{2\alpha}{\epsilon_0}\right\}ds
-\frac{1}{t_n}\int_0^{t_n} \mathbb P\{(u_s,v_s)\in K\}ds\\
&>1-\frac{\epsilon_0}{2}-\frac{\epsilon_0}{2}=1-\epsilon_0.
\end{align*}
Combining this and \eqref{E26}, we arrive at a contradiction: 
\begin{align*}
1=&\frac{1}{t_n}\int_0^{t_n} \mathbb P\left\{v_s<\frac{\epsilon}{2}\right\}ds+\frac{1}{t_n}\int_0^{t_n} \mathbb P\left\{v_s\geq \frac{\epsilon}{2}ds\right\}\\
>& 1-\epsilon_0 +\epsilon_0 =1. 
\end{align*}
Therefore,  \eqref{E28} holds true.

Let us fix a nonnegative function $ g\in C_0(\overline{\mathbb R^2_+})$ such that
\begin{align*}
g(x,y)=
\begin{cases}
 1,& \hspace{1cm}  (x,y)\in K,\\
0, &\hspace{1cm}  (x,y)\in \overline{\mathbb R^2_+}\backslash K_1,
\end{cases}
\end{align*} 
where $K_1\supset K$ is some bounded open set of $\overline{\mathbb R^2_+}.$ 
 In view of  \eqref{E28}, we have
\begin{align*}
\int_0^t P_sg(x,y)ds&=\int_0^t\int_{\overline{\mathbb R^2_+}} g(\xi,\eta)P(s, x,y, d\xi d\eta)ds\\
&\geq \int_0^t\int_{K} g(\xi,\eta)P(s, x,y, d\xi d\eta) ds\\
&= \int_0^t \mathbb P\{(u_s,v_s)\in K\}ds \to \infty \hspace{1cm} \text { as } t\to \infty.
\end{align*}
Thanks to  Theorem \ref{foguel}, we conclude that there exists a Borel invariant measure $\nu$ on $\overline{\mathbb R^2_{+}}$ for  $\{P_t\}_{0\leq t<\infty}$ such that $\nu(K)>0.$  By Theorems \ref{T1}--\ref{T2}, $\nu$ concentrates on some domain of $\mathbb R^2_+\cap \{(u,v); u\leq M_0\}$.   The proof is now complete.
\end{proof}

\section{Decline of forest} \label{Sec4}
In this section, we  show decline of the forest when either the mortality $h$ of old age class trees or the  intensity $\sigma$ of noise is large. More precisely, if   either 
$$h \geq \min\left\{\frac{\rho f}{c+f}, \frac{f(\rho+2abM^*)}{ab^2+c+f}\right\} $$
 or 
$$ \sigma^2>\frac{(\rho+c-h)^2}{2c},$$
 then the forest  falls into the decline. Here, $M^*$ is defined in Theorem \ref{T2}--(i).

\begin{theorem}\label{T5}
Let $(u_t,v_t)$ be the solution of \eqref{E1} with $(u_t,v_t)|_{t=0}=(u_0,v_0)\in \overline{\mathbb R^2_+}\backslash \{(0,0)\}$. Assume that $h\geq \min\left\{\frac{\rho f}{c+f}, \frac{f(\rho+2abM^*)}{ab^2+c+f}\right\}.$
Then, as $t\to \infty$, $u_t$ and $v_t$ converge to $0$ in expectation, i.e.
\begin{equation} \label{E40}
\lim_{t\to\infty} \mathbb Eu_t=\lim_{t\to\infty} \mathbb Ev_t=0.
\end{equation}
 In particular, $u_t$ and $v_t$ converge to $0$ in probability:
\begin{equation} \label{decline1}
\lim_{t\to\infty}\mathbb P\{u_t\geq C\}=\lim_{t\to\infty}\mathbb P\{v_t\geq C\}=0, \hspace{1cm} C>0.
\end{equation}
 Furthermore,  
\begin{equation} \label{decline2}
\lim_{t\to \infty} \mathbb P\{(u_t,v_t)\in A\}=0\quad \quad \text{ for any compact set } A\subset \mathbb R^2_{+}.
\end{equation}
\end{theorem}

\begin{proof}
Let us first prove that $u_t$ and $v_t$ converge to $0$ in expectation. Consider the two cases of the mortality $h$.

{\bf Case 1:} $h\geq \frac{\rho f}{c+f}$.  It follows from \eqref{E1} that
\begin{align*}
\begin{cases}
\mathbb E u_t\leq  u_0+ \int_0^t [\rho \mathbb Ev_s-(c+f)\mathbb Eu_s] ds,\\
\mathbb E v_t= v_0+ \int_0^t [f\mathbb Eu_s-h\mathbb Ev_s] ds.
\end{cases}
\end{align*}
Since the functions $\varpi_1$ and $\varpi_2$  defined by
$$\varpi_1(X, Y)=\rho Y-(c+f)X, \quad \varpi_2(X,Y)=fX-hY$$
 are non-decreasing with respect to arguments $Y$ and $X,$ respectively, the comparison theorem applied to the latter system provides that 
\begin{equation} \label{E9.5}
\begin{aligned}
\begin{cases}
\mathbb Eu_t\leq x_t, &\hspace{1cm}  0\leq t<\infty, \\
\mathbb E v_t\leq y_t, & \hspace{1cm}  0\leq t<\infty, 
\end{cases}
\end{aligned}
\end{equation}
 where $(x_t,y_t)$ is the positive solution to the linear system:
\begin{equation}\label{E10}
\begin{cases}
\frac{dx_t}{dt}= \rho y_t-(c+f)x_t,\\
\frac{dy_t}{dt}= fx_t-hy_t
\end{cases}
\end{equation}
with $(x_t, y_t)|_{t=0}=( u_0, v_0 ).$ 

From this system and a fact  that $h\geq \frac{\rho f}{c+f}$, we observe that 
\begin{align*}
\begin{cases}
\frac{d(hx_t+\rho y_t)}{dt}=[ \rho f - h (c+f)]x_t\leq 0,\\
\frac{d[fx_t+(c+f) y_t]}{dt}=[ \rho f - h (c+f)]y_t\leq  0.
\end{cases}
\end{align*}
Hence,  $hx_t+\rho y_t$ and $fx_t+(c+f) y_t$ are non-increasing as $t$ increases. As a consequence,  there exist two nonnegative constants $\beta_1$ and $\beta_2$ such that
\begin{equation*}
\begin{cases}
\lim_{t\to\infty}(hx_t+\rho y_t)=\beta_1, \\
 \lim_{t\to\infty}[fx_t+(c+f) y_t]=\beta_2.
\end{cases}
\end{equation*}
It is then seen that    
\begin{equation*}
\begin{cases}
\lim_{t\to\infty} x_t=\frac{\beta_1(c+f)-\beta_2\rho}{h(c+f)-\rho f}, \\
 \lim_{t\to\infty} y_t=\frac{\beta_2h-\beta_1f}{h(c+f)-\rho f}.
\end{cases}
\end{equation*}
This means that  
$(\frac{\beta_1(c+f)-\beta_2\rho}{h(c+f)-\rho f},\frac{\beta_2h-\beta_1f}{h(c+f)-\rho f})$
 is a stationary solution of \eqref{E10}. Substituting this solution for $(x_t,y_t)$  in  \eqref{E10}, we obtain that
\begin{align*}
\begin{cases}
\rho   (\beta_2h-\beta_1f)-(c+f)  [\beta_1(c+f)-\beta_2\rho]=0,\\
f[\beta_1(c+f)-\beta_2\rho] -h  (\beta_2h-\beta_1f)=0 .
\end{cases}
\end{align*}
Solving this system of algebraic equations,  we arrive at 
$$\beta_1=\beta_2=0.$$
 Hence, 
\begin{equation}  \label{E9.6}
\lim_{t\to\infty} x_t=\lim_{t\to\infty} y_t=0.
\end{equation}

Combining \eqref{E9.5} and \eqref{E9.6}, we conclude that 
$$\lim_{t\to\infty}\mathbb Eu_t=\lim_{t\to\infty}\mathbb Ev_t=0.$$

{\bf Case 2:} $h\geq  \frac{f(\rho+2abM^*)}{ab^2+c+f}$.  
From  \eqref{E1} and Theorem \ref{T2}--(i), we have  
\begin{align*}
\begin{cases}
\begin{aligned}
\mathbb E u_t&\leq  u_0+ \int_0^t [\rho \mathbb Ev_s+2ab \mathbb E(u_sv_s)-(ab^2+c+f)\mathbb Eu_s] ds\\
&\leq u_0+ \int_0^t [(\rho+2abM^*) \mathbb Ev_s-(ab^2+c+f)\mathbb Eu_s] ds,
\end{aligned}\\
\mathbb E v_t= v_0+ \int_0^t [f\mathbb Eu_s-h\mathbb Ev_s] ds.
\end{cases}
\end{align*}
Using the same argument as in Case 1, we conclude that $$\lim_{t\to\infty}\mathbb Eu_t=\lim_{t\to\infty}\mathbb Ev_t=0.$$

Let us now verify \eqref{decline1} and \eqref{decline2}.  For any  $0<c_1<c_2, 0<d_1<d_2,$  
$$\mathbb P\{(u_t,v_t)\in [c_1, c_2]\times [d_1,d_2]\}\leq \mathbb P\{u_t\geq c_1\},$$
 and
$$\mathbb P\{u_t\geq c_1\}\leq \frac{1}{c_1}\mathbb E u_t, \quad \mathbb P\{v_t\geq c_1\}\leq \frac{1}{c_1}\mathbb E v_t.$$
This together with \eqref{E40} derives  \eqref{decline1} and \eqref{decline2}. It completes the proof of the theorem.
\end{proof}

Under somewhat stronger assumptions than those of Theorem \ref{T5}, we can show  almost sure convergence  of $u_t$ and $ v_t$ to $0$. Consider two functions $F_1$ and $F_2$ defined by 
\begin{align}
F_1(x)=&f^2 x^4+ 2f(\sigma^2+h-c-f) x^3+[(c+f-h)^2 \notag \\
&-2\rho f -2(c+f) \sigma^2] x^2+2\rho (c+f-h)x +\rho^2,\label{E31}
\end{align}
and
\begin{equation}\label{E31.1}
F_2(x)= fx^2-(c+f+h) x+\rho.
\end{equation}
Assume that either
\begin{equation} \label{Hypothesis1}
\inf_{ x\in(0, \frac{c+f}{f})} F_1(x)<0,
\end{equation}
or
\begin{equation} \label{Hypothesis2}
\begin{aligned}
& \frac{2\rho}{\sigma^2+2h}<\frac{c+f}{f}
\text{ and there exists  } \lambda \text{ such that }   \\
&\frac{2\rho}{\sigma^2+2h}<\lambda<\frac{c+f}{f},  
 F_1(\lambda)>0 \text{ and }  F_2(\lambda)<0
\end{aligned}
\end{equation}
holds true. 
Then, the following theorem shows such convergence.

\begin{theorem}\label{T6}
Let $(u_t,v_t)$ be the solution of \eqref{E1} with $(u_t,v_t)|_{t=0}=(u_0,v_0)\in \overline{\mathbb R^2_+}\backslash \{(0,0)\}$. 
Under \eqref{Hypothesis1} or \eqref{Hypothesis2},  $\lim_{t\to\infty} u_t=\lim_{t\to\infty} v_t=0 $ \,   a.s.
\end{theorem}
 \begin{proof}
We  again use the function $Q$ defined by $Q(u,v)=\log (u+\kappa v) $ as in  the proof of Theorem \ref{T8}, where $\kappa $ is a positive constant that will be fixed below.  

Let us first  show that under \eqref{Hypothesis1} or \eqref{Hypothesis2}, there exists a small $\epsilon>0$ such that 
$$[LQ](u,v)\leq -\frac{\epsilon}{2} \hspace{1cm} \text{ for all } (u,v)\in {\mathbb R^2_+},$$
   where $[LQ]$ is defined in  \eqref{E22.1}.
Indeed, it is easily seen that a sufficient condition for this (in fact, it is also a necessary condition)   is  that there exists  $\epsilon>0$ such that
\begin{equation}\label{E12}
\begin{aligned}
F(u,v)=&[2(c+f-\kappa f)-\epsilon]u^2 -2[\kappa (\kappa f-c-f)+\rho-\kappa h+\kappa \epsilon]v u\\
&+[\sigma^2\kappa^2+2\kappa (\kappa h-\rho)-\kappa^2 \epsilon]v^2 \geq 0\hspace{1.5cm} \text{ for all } (u,v)\in {\mathbb R^2_+}.
\end{aligned}
\end{equation}
If  \eqref{Hypothesis1} takes place, choose 
$\kappa$ such that
$$0<\kappa< \frac{c+f}{f} \quad \text{ and } \quad  F_1(\kappa)<0.$$
It is then easily seen that  there exists a  small $\epsilon>0$ such that the quadratic equation $F(u,v)=0$ in  $u$ has a non-positive discriminant for all  $v\geq 0.$ This implies \eqref{E12}. In the meantime, if  \eqref{Hypothesis2} takes place, choose $\kappa=\lambda$ in \eqref{Hypothesis2}. Similarly, 
it is  seen that  there exists $\epsilon>0$ such that the  equation $F(u,v)=0$ has non-positive  two  solutions for all $v\geq 0$. This also derives \eqref{E12}.

Let us now verify that 
$$\lim_{t\to\infty} Q(u_t,v_t)=-\infty \hspace{1cm} \text{ a.s.} $$
It follows from \eqref{E22} that  
\begin{align}\label{E13}
\frac{1}{t}Q(u_t,v_t)&=\frac{1}{t}Q(u_0,v_0)+\frac{1}{t}\int_0^t[LQ](u_s,v_s)ds+\frac{1}{t} \int_0^t\frac{ \sigma \kappa v_s}{u_s+\kappa v_s}dw_s\notag \\
&\leq \frac{1}{t}Q(u_0,v_0) -\epsilon + \frac{1}{t}\int_0^t\frac{ \sigma \kappa v_s}{u_s+\kappa v_s}dw_s, \hspace{1cm} 0<t<\infty.
\end{align}
On account of \eqref{E25.1} and \eqref{E25.2}, 
$$\lim_{t\to\infty} \frac{1}{t}\int_0^t\frac{ \sigma \kappa v_s}{u_s+\kappa v_s}dw_s=0 \hspace{1cm} \text{ a.s. }$$
Hence, taking the limit as $t\to \infty$ of  both the hand sides of \eqref{E13}, we  observe that
$$\limsup_{t\to\infty}\frac{1}{t}Q(u_t,v_t) \leq -\frac{\epsilon}{2} \hspace{1cm} \text{ a.s.} $$
This implies that 
 $\lim_{t\to\infty} Q(u_t,v_t)=-\infty$ a.s.  Thus,   
$$\lim_{t\to\infty} u_t=\lim_{t\to\infty} v_t=0 \hspace{1cm} \text{ a.s.} $$
The proof is complete.
\end{proof}

\begin{remark}
It is possibly seen that if  $F_1(1)<0$, then  \eqref{Hypothesis1} takes place. 
After some simple calculations on the inequality  $F_1(1)<0$,  we arrive at this condition:
$$\sigma^2>\frac{(\rho+c-h)^2}{2c}.$$
According to Theorem \ref{T6}, we therefore conclude that  a  noise  with  large intensity    causes  decline  of  the forest.
\end{remark}

\section{Numerical examples } \label{Sec5}
Let us  exhibit some numerical examples for sustainability of the forest and  possibility of decline. For the computations, we used a scheme of order  $1.5$ (see, e.g.,  \cite{KPS}).

\subsection {Sustainability of forest} 
In the system \eqref{E1}, set $a=2, b=1, c=2.5, f=4, h=1, \rho=5, \sigma=0.5,$ and take  initial value $(u_0,v_0)=(2,1).$

Figure \ref{Fig2}  gives  sample trajectories of $u $ and $v$ in the phase space and  in time.

Figure \ref{Fig3} plots  points $(u_T,v_T)$ of $10^4$ sample trajectories  of $(u,v)$ at time $T=1000$.

By computing $10^3$ sample trajectories of $(u,v)$, Figure \ref{Fig4}  shows  a graph of the expectation of tree densities of young  and old age classes.

Figure \ref{Fig5} gives a sample trajectory of two processes $I$ and $J$ defined by
 $$I(t)=\frac{1}{t}\int_0^t u_sds, \quad  
J(t)=\frac{1}{t}\int_0^t v_sds.$$ 

Figure \ref{Fig6} demonstrates a trajectory of two probability functions $R$ and $S$  defined by 
$$
 \begin{cases}
R(t)=\mathbb P\{(u_t,v_t)\in A; (u_0,v_0)=(2,1)\}, \\
S(t)=\mathbb P\{(u_t,v_t)\in A; (u_0,v_0)=(3,4)\},
\end{cases}
$$
along $t\in [50, 100]$, where $ A=[0.5,30]\times[0.1,20].$
These functions are  calculated on the basis of  $2000$ sample trajectories of $(u_t,v_t)$ corresponding to each of the two initial values.

  \begin{figure}[H] 
\centerline{\psfig{file=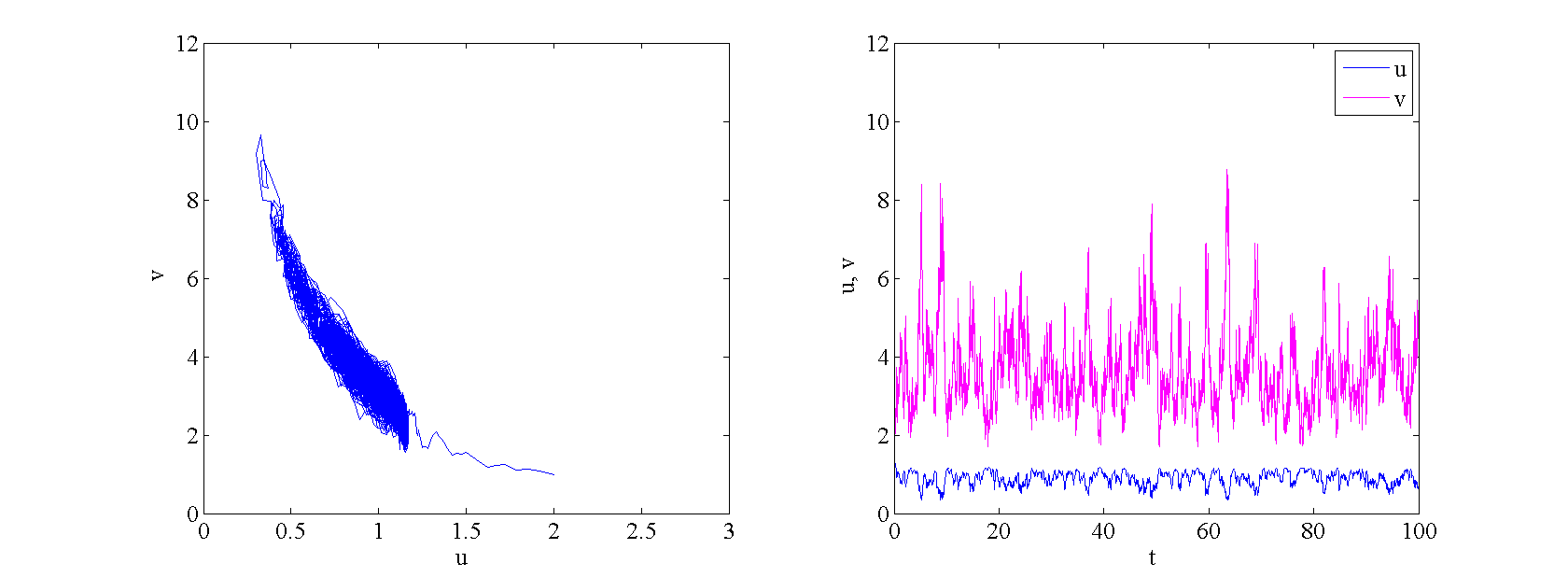,width=5.5in}}
\caption{ Sample trajectories  of $u_t$ and $v_t$ of \eqref{E1} with parameters: $a=2, b=1, c=2.5, f=4, h=1, \rho=5, \sigma=0.5$ and  initial value $(u_0,v_0)=(2,1).$ The left figure illustrates a sample trajectory of $(u_t,v_t)$ in the phase space; the right figure illustrates  sample trajectories of $u_t$ and $v_t$ along  $t\in [0,100]$.}
\label{Fig2}
\end{figure}

\begin{figure}[H]
\centerline{\psfig{file=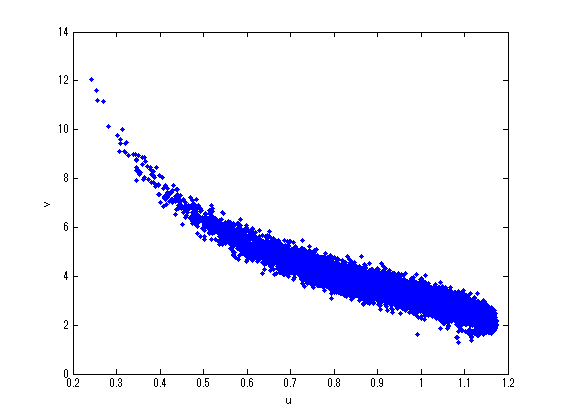,width=5.5in}}
 \caption{Distribution of $(u_t,v_t)$ of \eqref{E1} at $t=10^3$. The parameters and initial value are taken as in  the legend of Fig.~\ref{Fig2}.}
\label{Fig3}
\end{figure}

\subsection{Decline of forest} 
First, set $a=3,  b=4, c=5, f=6, h=2, \rho=7, \sigma=4$  and take $(u_0,v_0)=(4,3)$. Figure \ref{Fig7} gives  sample trajectories of $u $ and $v$ in the phase space and in time.

Second, set $a=3,    b=4, c=5, f=6, h=3.82, \rho=7, \sigma=0.25$ and take  $(u_0,v_0)=(4,3).$ By computing $5\times 10^2$ sample trajectories of $(u, v)$,  Figure  \ref{Fig8}   shows  a graph of  expectation of tree densities of young and old age classes.

\begin{figure}[H]
\centerline{\psfig{file=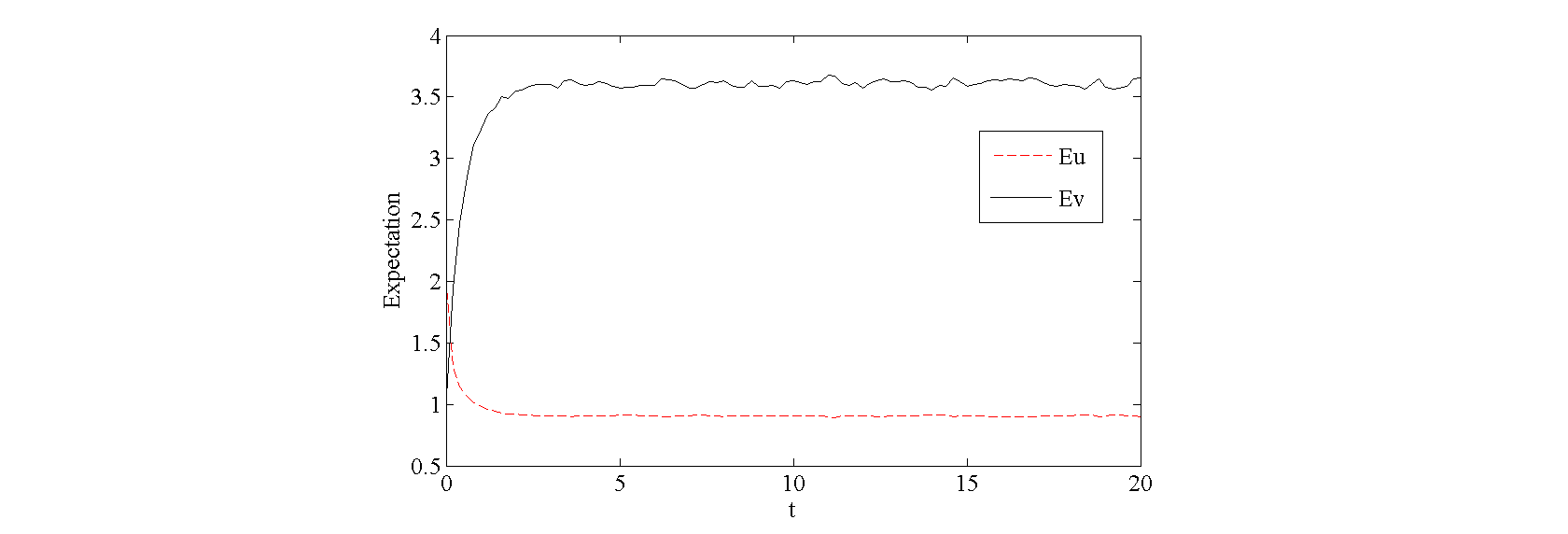,width=8in}}
 \caption{Graphs of  $\mathbb Eu$ and $\mathbb Ev$ along $ t\in [0,20]$. The parameters and   initial value are taken as in the legend of  Fig.~\ref{Fig2}.}
\label{Fig4}
\end{figure}

\begin{figure}[H]
\centerline{\psfig{file=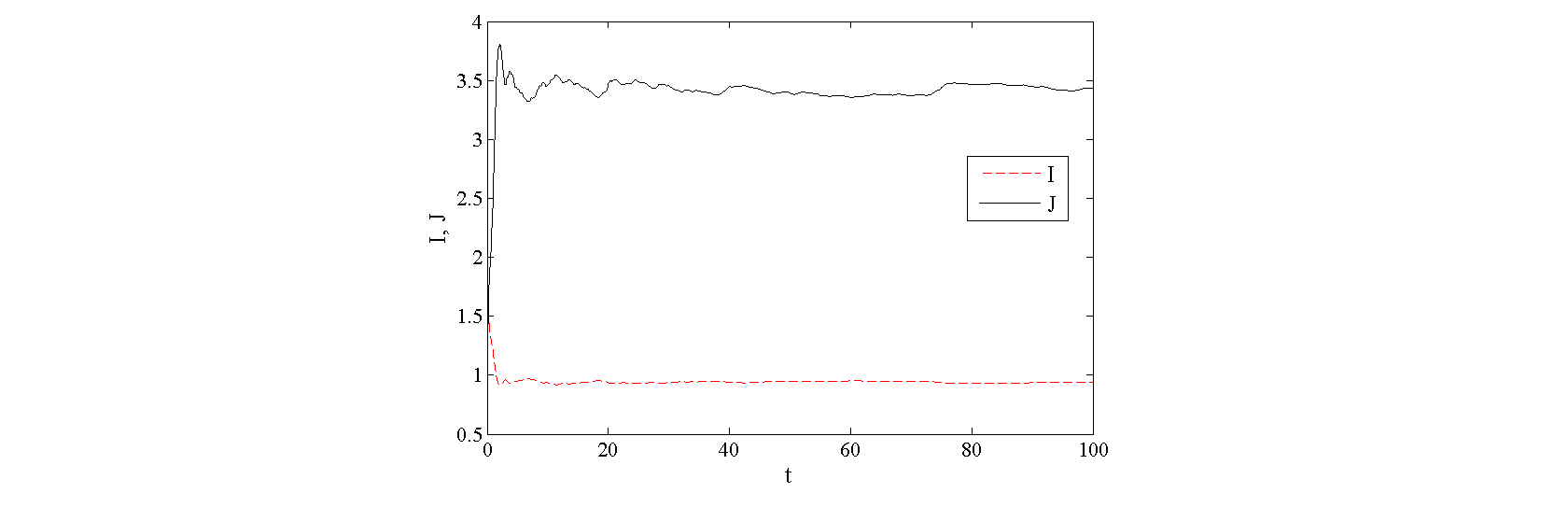,width=9in}}
 \caption{Sample trajectory  of  two processes $I$ and $J$ defined by $I(t)=\frac{1}{t}\int_0^t u_sds$ and $J(t)=\frac{1}{t}\int_0^t v_sds$ along $ t\in [0,100].$ The parameters and   initial value are  taken as in the legend of  Fig.~\ref{Fig2}.}
\label{Fig5}
\end{figure}

\begin{figure}[H]
\centerline{\psfig{file=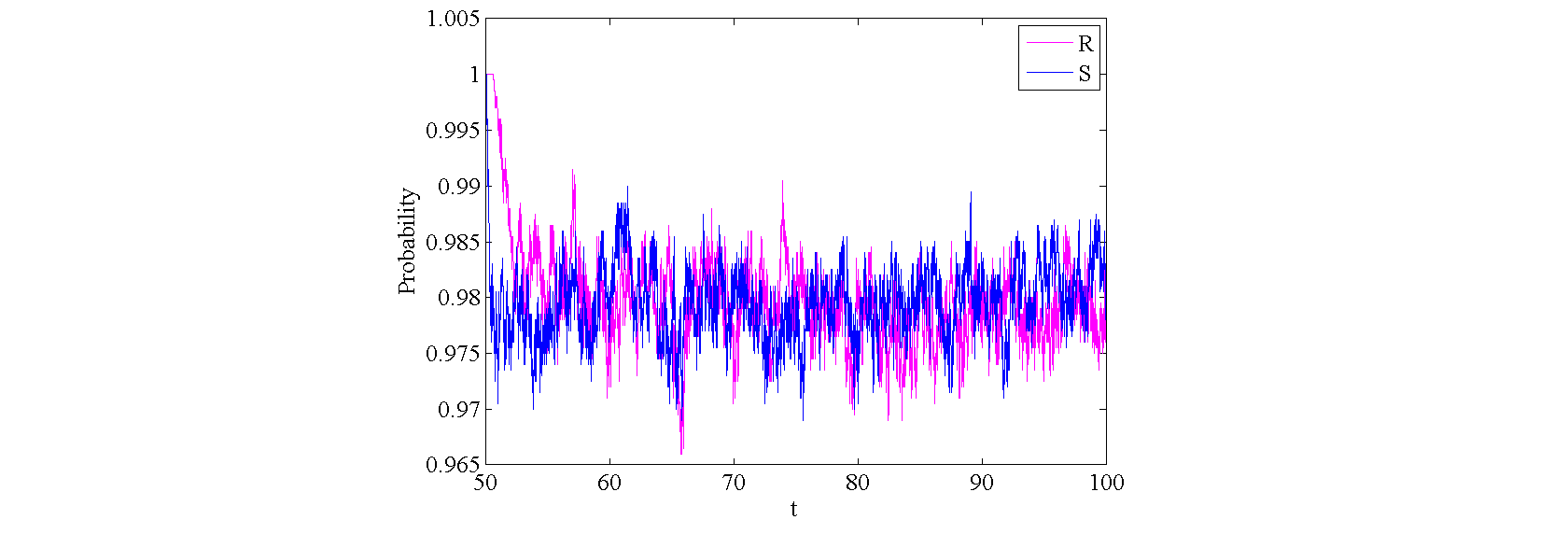,width=10in}}    
 \caption{Graph of  probability functions $R$ and $S$  defined by $R(t)=\mathbb P\{(u_t,v_t)\in A; (u_0,v_0)=(2,1)\}$ and  $S(t)=\mathbb P\{(u_t,v_t)\in A; (u_0,v_0)=(3,4)\}$  along $t\in [50,100],$ 
where $A=[0.5,30]\times[0.1,20]$ and the parameters of \eqref{E1}  are taken as in the legend of  Fig.~\ref{Fig2}. These functions are  calculated on the basis of  $2000$ sample trajectories of $(u_t,v_t)$ corresponding to each initial value.  }
\label{Fig6}
\end{figure}

 \begin{figure}[H]
\centerline{\psfig{file=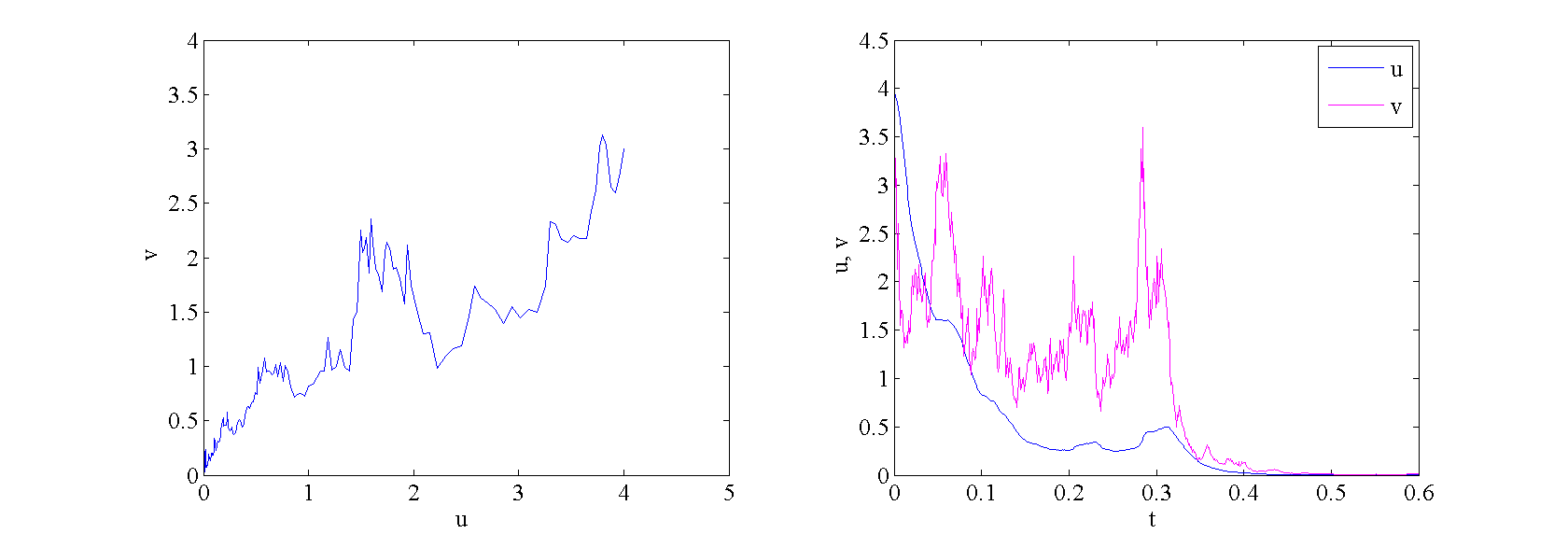,width=5.6in}}
 \caption{ Decline of forest under the effect of noise with large intensity $\sigma$. Here, $a=3,  b=4, c=5, f=6, h=2, \rho=7, \sigma=4$  and  initial value $(u_0,v_0)=(4,3)$. The left figure is a sample trajectory of $(u_t,v_t)$  in the phase space; the right figure is a sample trajectory of $u$ and $v$ along  $t\in [0,1]$.}
\label{Fig7}
\end{figure}

 \begin{figure}[H]
\centerline{\psfig{file=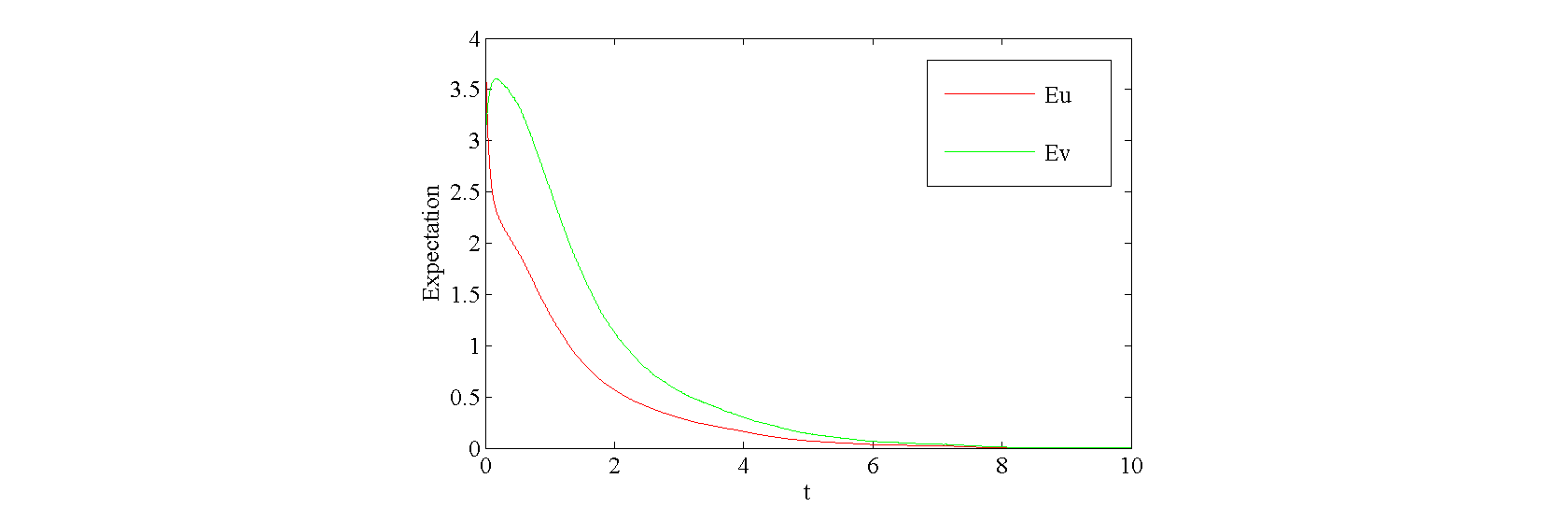,width=9in}}
 \caption{ Decline of forest when the  mortality $h$ of old  trees is large. Here,  $a=3,    b=4, c=5, f=6, h=3.82, \rho=7, \sigma=0.25$ and  initial value  $(u_0,v_0)=(4,3).$ The  figure  gives  a  graph of  $\mathbb Eu$ and $\mathbb Ev$ along $t\in [0,10]$.}
\label{Fig8}
\end{figure}


\begin{thebibliography}{00}

\bibitem {MYA} M.~Ya.~Antonovsky,  
 \emph{Impact of the factors of the environment on the dynamics of population (mathematical model)},
 in   Proc. Soviet-American Symp. ``Comprehensive Analysis of the Environment'', Tbilisi 1974, Leningrad: Hydromet,  (1975),  218--230.

\bibitem{Arnold} (MR0443083) L.~Arnold, 
 ``Stochastic Differential Equations: Theory and Applications," 
 Wiley, New York, 1972.

\bibitem{Chuan1} (MR2356122) L.~H.~Chuan and A.~Yagi,  
\emph{Dynamical system for forest kinematic model,}
 Adv. Math. Sci. Appl.,  {\bf 16 } (2006),  393--409. 

\bibitem{Chuan2} (MR2297947) L.~H.~Chuan, T.~Tsujikawa and A.~Yagi,   
\emph{Asymptotic behavior of solutions for forest kinematic model,}
 Funkcial. Ekvac.,  {\bf 49 } (2006), 427--449.

\bibitem{Chuan3} (MR2471671) L.~H.~Chuan, T.~Tsujikawa and  A.~Yagi,  
\emph{Stationary solutions to forest kinematic model,}
 Glasg. Math. J., {\bf 51 } (2009),  1--17.

 \bibitem {Foguel} (MR0372154) S.~R.~Foguel,   
\emph{The ergodic theory of positive operators on continuous functions,} 
 Ann. Scuola Norm. Sup. Pisa,   {\bf 27} (1973), 19--51.

\bibitem{Friedman} (MR0494491) A.~Friedman, 
  ``Stochastic Differential Equations and Applications,"
 Academic Press, New York, 1976.

\bibitem{WI} (MR0637061)  N.~Ikeda and   S.~Watanabe,   
 ``Stochastic Differential Equations and Diffusion Processes,"
North-Holland, Tokyo, 1981.

\bibitem{IS}  (MR1121940) I.~Karatzas and    S.~E.~Shreve,   
  ``Brownian Motion and Stochastic Calculus,"
 Springer-Verlag, Berlin, 1991.

\bibitem{KPS}  (MR1260431) P.~E.~Kloeden,  E.~Platen and   H.~Schurz, 
     ``Numerical Solution of SDE through Computer Experiments,"
 Springer-Verlag, Berlin, 1994.

\bibitem{KABA} (MR1266986)  Yu.~A.~Kuznetsov,  M.~Ya.~Antonovsky, V.~N.~Biktashev  and   E.~A.~Aponina, 
 \emph{A cross-diffusion model of forest boundary dynamics,}
J. Math. Biol.,   {\bf 32} (1994),  219--232.

\bibitem{Mao} (MR2380366)  X.~Mao,   
  ``Stochastic Differential Equations and Applications,"
 2nd edition, Horwood, Chichester, 2008.

\bibitem {Mic} (MR0265559) L.~Michael, 
  \emph{Conservative Markov processes on a topological space,}
  Isr. J. Math.,  {\bf 8} (1970), 165--186.

\bibitem{LinhTon} (MR2823878) L.~T.~H.~Nguyen and   T.~V.~T\d{a}, 
 \emph{Dynamics of a stochastic ratio-dependent predator-prey model,}
 Anal. Appl. (Singap.),  {\bf  9} (2011), 329--344.


\bibitem{Shirai} (MR2348471) T.~Shirai, L.~H.~Chuan and  A.~Yagi,  
 \emph{Asymptotic behavior of solutions for forest kinematic model under Dirichlet conditions,} 
 Sci. Math. Jpn., {\bf 66 } (2007),  289--301. 


\bibitem{TonLinhYagi} (MR3150970)  T.~V.~T\d{a},  L.~T.~H.~Nguyen and A.~Yagi, 
  \emph{Flocking and non-flocking behavior in a stochastic Cucker-Smale system,}
 Anal. Appl. (Singap.),  {\bf 12} (2014),   63--73.


\bibitem{AA} (MR2573296) A.~Yagi,  
 ``Abstract Parabolic Evolution Equations and their Applications,"
  Springer-Verlag, Berlin, 2010.
\end{thebibliography}
\end{document}